\documentclass{amsart}

\usepackage{amssymb,amsthm,amsmath,amsxtra}
\usepackage[all]{xy}
\usepackage{verbatim}

\numberwithin{equation}{section}

\theoremstyle{plain}
\newtheorem{thm}[equation]{Theorem}
\newtheorem{prop}[equation]{Proposition}
\newtheorem{lem}[equation]{Lemma}
\newtheorem{cor}[equation]{Corollary}

\newtheorem{defn}[equation]{Definition}
\newtheorem*{prop*}{Proposition}
\newtheorem*{thm*}{Theorem}
\newtheorem*{thma*}{Theorem A}
\newtheorem*{thmb*}{Theorem B}
\newtheorem*{thmc*}{Theorem C}

\theoremstyle{remark}

\newtheorem{exm}[equation]{Example}
\newtheorem{rmk}[equation]{Remark}

\newenvironment{enumroman}
{\begin{enumerate}}
{\end{enumerate}}

\entrymodifiers={+!!<0pt,\fontdimen22\textfont2>}

\DeclareMathOperator{\ann}{ann}

\DeclareMathOperator{\gdeg}{gdeg}

\DeclareMathOperator{\End}{End}

\DeclareMathOperator{\Hom}{Hom}

\DeclareMathOperator{\nrd}{nrd}

\DeclareMathOperator{\rk}{rk}

\DeclareMathOperator{\Spec}{Spec}

\DeclareMathOperator{\trd}{trd}
\DeclareMathOperator{\Tr}{Tr}

\newcommand{\F}{\mathbb F}

\newcommand{\Z}{\mathbb Z}

\newcommand{\eps}{\epsilon}

\newcommand{\fraka}{\mathfrak{a}}

\newcommand{\frakm}{\mathfrak{m}}
\newcommand{\frakp}{\mathfrak{p}}
\newcommand{\frakq}{\mathfrak{q}}

\SelectTips{eu}{12}

\begin{document}

\title[Rings of low rank]{Rings of low rank with a standard involution}

\author{John Voight}
\address{Department of Mathematics and Statistics, University of Vermont, 16 Colchester Ave, Burlington, VT 05401, USA}
\email{jvoight@gmail.com}
\date{\today}

\begin{abstract}
We consider the problem of classifying (possibly noncommutative) $R$-algebras of low rank over an arbitrary base ring $R$.  We first classify algebras by their degree, and we relate the class of algebras of degree $2$ to algebras with a standard involution.  We then investigate a class of exceptional rings of degree $2$ which occur in every rank $n \geq 1$ and show that they essentially characterize all algebras of degree $2$ and rank $3$.
\end{abstract}

\maketitle

Let $R$ be a commutative Noetherian ring (with $1$) which is connected, so that $R$ has only $0,1$ as idempotents (or equivalently that $\Spec R$ is connected).  Let $B$ be an algebra over $R$, an associative ring with $1$ equipped with an embedding $R \hookrightarrow B$ of rings (mapping $1 \in R$ to $1 \in B$) whose image lies in the center of $B$; we identify $R$ with its image in $B$.  Assume further that $B$ is a finitely generated, projective $R$-module.  Recall that a finitely generated module is projective if and only if it is locally free; we define the \emph{rank} of $B$ to be the common rank of its localizations.

The problem of classifying algebras $B$ of low rank has an extensive history.  The identification of quadratic rings over $\Z$ by their discriminants is classical and goes back as far as Gauss.  Commutative rings of rank at most $5$ over $R=\Z$ have been classified by Bhargava \cite{Bhargava1}, building on work of others; this beautiful work has rekindled interest in the subject and has already seen many applications.  Progress on generalizing these results to arbitrary commutative base rings $R$ (or even arbitrary base schemes) has been made by Wood \cite{MWood}.  A natural question in this vein is to consider noncommutative algebras of low rank, and in this article we treat algebras of rank at most $3$.  

The category of $R$-algebras (with morphisms given by isomorphisms) has a natural decomposition by degree.  The \emph{degree} of an $R$-algebra $B$, denoted $\deg_R(B)$, is the smallest positive integer $n$ such that every $x \in B$ satisfies a monic polynomial of degree $n$.  Any \emph{quadratic} algebra $B$, i.e.\ an algebra of rank $2$, is necessarily commutative (see Lemma \ref{quad2}) and has degree $2$.  Moreover, a quadratic algebra has a unique $R$-linear (anti-)involution $\overline{\phantom{x}}:B \to B$ such that $x\overline{x} \in R$ for all $x \in B$, which we call a \emph{standard involution}.  

The situation is much more complicated in higher rank.  In particular, the degree of $B$ does not behave well with respect to base extension (Example \ref{notbehavewell}).  We define the \emph{geometric degree} of $B$ to be the maximum of $\deg_S(B \otimes_R S)$ with $R \to S$ a homomorphism of (commutative) rings.  Our first main result is as follows (Corollary \ref{cor10}).

\begin{thma*}
Let $B$ be an $R$-algebra and suppose there exists $a \in R$ such that $a(a-1)$ is a nonzerodivisor.  Then the following are equivalent.
\begin{enumroman}
\item $B$ has degree $2$;
\item $B$ has geometric degree $2$;
\item $B \neq R$ has a standard involution.
\end{enumroman}
\end{thma*}

Note that if $2$ is a nonzerodivisor in $R$ then we can take $a=-1$ in the above theorem.

In view of Theorem A, it is natural then to consider the class of $R$-algebras $B$ equipped with a standard involution which is then necessarily unique (Corollary \ref{uniqstd}).  For such an algebra $B$, we define the \emph{reduced trace} $\trd:B \to R$ by $x \mapsto x+\overline{x}$ and the \emph{reduced norm} by $\nrd:B \to R$ by $x \mapsto x\overline{x}$; then every element $x \in B$ satisfies the polynomial $\mu(x;T)=T^2-\trd(x)T+\nrd(x)$.

Commutative algebras with a standard involution can be easily characterized: for example, if $2$ is a nonzerodivisor in $R$ and $B$ is a commutative $R$-algebra with a standard involution, then either $B$ is a quadratic algebra or $B$ is a quotient of an algebra of the form $R[x_1,\dots,x_n]/(x_1,\dots,x_n)^2$ (more generally, see Proposition \ref{commutativecase}).

There is a natural class of noncommutative algebras equipped with a standard involution which occur in every rank $n \geq 1$, defined as follows.  Let $M$ be a projective $R$-module of rank $n-1$ and let $t:M \to R$ be an $R$-linear map.  Then we give the $R$-module $B=R \oplus M$ the structure of an $R$-algebra by defining the multiplication rule $xy = t(x)y$ for $x,y \in M$.  The map $x \mapsto \overline{x}=t(x)-x$ is a standard involution on $B$.  An \emph{exceptional ring} is an $R$-algebra $B$ with the property that there is a left ideal $M \subset B$ such that $B=R \oplus M$ and the map $M \to \Hom_R(M,B)$ given by left multiplication factors through a linear map $t:M \to R$.  

Our second main result (Theorem \ref{rank3}) is as follows.

\begin{thmb*}
An $R$-algebra $B$ of rank $3$ has a standard involution if and only if it is an exceptional ring.
\end{thmb*}

The results of this paper will be further used in an upcoming work \cite{Voightquat} which investigates algebras of rank $4$ with a standard involution, in an attempt to characterize quaternion rings over an arbitrary base ring.

This article is organized as follows.  We begin (\S 1) with some preliminary notions and define the degree of an algebra.  We then explore the relationship between algebras of degree $2$ and those with a standard involution and then prove Theorem A (\S 2).  Next, we investigate the class of commutative algebras with a standard involution and define exceptional rings (\S 3).  We then classify algebras of rank $3$, relating them to certain endomorphism rings of flags and prove Theorem B (\S 4).  

The author would like to thank Hendrik Lenstra for his suggestions and comments which helped to shape this research.  We are particularly indebted to Melanie Wood who made many helpful remarks and suggestions.  We would also like to thank David Speyer and the referee.

\section{Degree} \label{secdegree}

In this section, we discuss the notion of the degree of an algebra, generalizing the notion from that over a field.  We refer the reader to Scharlau \cite[\S 8.11]{Scharlau} for an alternative approach.  

Throughout this article, let $R$ be a commutative, connected Noetherian ring and let $B$ be an algebra over $R$, which as in the introduction is defined to be an associative ring with $1$ equipped with an embedding $R \hookrightarrow B$ of rings.  We assume further that $B$ is finitely generated, projective $R$-module.  For a prime $\frakp$ of $R$, we denote by $R_\frakp$ the localization of $R$ at $\frakp$; we abbreviate $B_\frakp = B \otimes_R R_\frakp$ and for $x \in B$ we write $x_\frakp = x \otimes 1 \in B_\frakp$.  Since $B$ is projective, we have that $B_\frakp$ is locally free of finite rank $n$, which we suppose throughout is independent of $\frakp$, and we define the \emph{rank} of $B$ to be this common rank and denote $n=\rk_R(B)$.

\begin{rmk}
There is no loss of generality in working with connected rings, since for an arbitrary ring $R$ one has a statement for each of the connected components of $\Spec R$.  Furthermore, one may work with non-Noetherian rings by the process of Noetherian reduction, by finding a Noetherian subring $R_0 \subset R$ and an $R_0$-algebra $B_0$ such that $B_0 \otimes_{R_0} R \cong B$.
\end{rmk}

\begin{rmk}
For the questions we consider herein, we work (affinely) with algebras over base rings.  If desired, one could without difficulty extend our results to an arbitrary (separated) base scheme by the usual patching arguments.
\end{rmk}

We begin with a preliminary lemma.

\begin{lem} \label{dirsummand}
$R$ is a direct summand of $B$.
\end{lem}

\begin{proof}
For every prime ideal $\frakp$ of $R$, there exists a basis for the algebra $B_\frakp/\frakp B_\frakp$ over the field $R_\frakp/\frakp R_\frakp$ which includes $1$, and by Nakayama's lemma this lifts to a basis for $B_\frakp$.  In particular, the quotient $B/R$ is locally free and finitely generated of constant rank (since $B$ is finitely generated over $R$, and $R$ is connected) hence projective, which implies that $B/R$ and hence $R$ is a direct summand of $B$.  
\end{proof}

Every element $x \in B$ satisfies a monic polynomial with coefficients in $R$ by the (generalized) Cayley-Hamilton theorem; indeed, by the ``determinant trick'', this polynomial has degree bounded by the minimal number of generators for $B$ as an $R$-module \cite[Theorem IV.17]{Macdonald} (see also the determinant-trace polynomial \cite[Section V.E]{Macdonald}).  In fact, one can extend the notion of characteristic polynomial directly as follows.

\begin{lem} \label{charpoly}
For every $x \in B$, there exists a unique monic polynomial $\chi(x;T) \in R[T]$ of degree $n=\rk(B)$ with the property that for every prime $\frakp$ of $R$, the characteristic polynomial of left multiplication by $x$ on $B_\frakp$ is equal to $\chi(x;T)_\frakp \in R_\frakp[T]$.  Moreover, we have $\chi(x;x)=0$.
\end{lem}

\begin{proof}
Let $x\in B$.  Since $B$ is projective, for each prime $\frakp$ of $R$ we have that $B_\frakp$ is free over $R_\frakp$ of rank $n$.  By the determinant trick, we see that $x_\frakp \in R_\frakp$ satisfies the characteristic polynomial $\chi_\frakp(x;T) \in R_\frakp[T]$ of left multiplication by $x_\frakp$ on $B_\frakp$, where $\chi_\frakp(x;T)$ is monic of degree $n$.  Therefore by standard patching arguments \cite[Proposition II.2.2]{Hartshorne} (see also the proof of Proposition \ref{quad2}), there exists a unique monic polynomial $\chi(x;T) \in R[T]$ such that $\chi(x;T)_\frakp = \chi_\frakp(x;T)$.  Finally, since $\chi(x;x)_\frakp=0 \in R_\frakp$ for all primes $\frakp$, we have that $\chi(x;x)=0 \in R$.
\end{proof}

\begin{defn}
The \emph{degree} of $x \in B$, denoted $\deg_R(x)$ (or simply $\deg(x)$ if the base ring $R$ is clear from context), is the smallest positive integer $n \in \Z_{>0}$ such that $x$ satisfies a monic polynomial of degree $n$ with coefficients in $R$.
\end{defn}

By Lemma \ref{charpoly}, we have $\deg_R(x) \leq \rk B$ for all $x \in B$.  Note that $\deg_R(x)=1$ if and only if $x \in R$.

For $x \in B$, denote by $R[x]$ the (commutative) $R$-subalgebra of $B$ generated by $x$, i.e., $R[x]=\bigcup_{d=0}^{\infty} Rx^d \subset B$.  

\begin{lem} \label{projfree}
Let $x \in B$.  Then the following are equivalent:
\begin{enumroman}
\item $R[x]$ is free as an $R$-module;
\item $R[x]$ is projective as an $R$-module;
\item $x$ satisfies a unique monic polynomial of minimal degree $\deg_R(x)$ with coefficients in $R$; 
\item The ideal $\{f(T) \in R[T] : f(x)=0\} \subset R[t]$ is principal and generated by a monic polynomial.
\end{enumroman}
If any one of these holds, then $\deg_R(x) = \rk_R R[x]$.  
\end{lem}

\begin{proof}
The lemma is clear if $x \in R$, so we may assume $x \not \in R$ or equivalently $\deg_R(x)>1$.

The statement (i) $\Rightarrow$ (ii) is trivial.  To prove (ii) $\Rightarrow$ (i), suppose that $R[x]$ is projective.  Let $\frakp$ be a prime ideal of $R$ and let $k=R_\frakp/\frakp R_\frakp$ be the residue field of $R_\frakp$.  Then $R[x] \otimes_R k = k[x]$ has a $k$-basis $1,x,\dots,x^{d-1}$ for some $d \in \Z_{>1}$.  By Nakayama's lemma, $1,\dots,x^{d-1}$ is a $R_\frakp$-basis for $R_\frakp[x]$.  Since $R$ is connected, the value of $d=\rk R_\frakp[x]$ does not depend on the prime ideal $\frakp$.  It follows that the surjective map $\bigoplus_{i=0}^{d-1} Re_i \to R[x]$ by $e_i \mapsto x^i$ is an isomorphism since it is so locally, and hence $R[x]$ is free.

To prove that (iii) $\Leftrightarrow$ (i), we note that if $f(T) \in R[T]$ is the unique monic polynomial of degree $d=\deg_R(x) \geq 2$ with $f(x)=0$, then $1,x,\dots,x^{d-1}$ is an $R$-basis for $R[x]$---indeed, if $a_{d-1}x^{d-1}+\dots+a_0=0$ with $a_i \in R$ then $g(T)=f(T)+a_{d-1}T^{d-1}+\dots+a_0$ has $g(x)=0$ so $f(T)=g(T)$ and $a_0=\dots=a_{d-1}=0$, and the converse follows similarly.

The equivalence (iii) $\Leftrightarrow$ (iv) follows similarly.
\end{proof}

\begin{cor} \label{projfreedeg2}
Suppose that $\deg_R(x)=2$.  Then $R[x]$ is projective if and only if $ax \not\in R$ for all $a \neq 0 \in R$, and this holds if $1,x$ belongs to a basis for $B$.
\end{cor}

\begin{exm}
Let $p$ be prime and let $B=R[\eps]/(\eps^2)$ with $R=\Z/p^2\Z$.  Then $R[\eps]=B$ is projective, but the element $x=p\eps$ satisfies $x^2=0$ as well as $px=0$, so $R[x]$ is not projective.
\end{exm}

If $R \to S$ is a ring homomorphism and $x \in B$, then we abbreviate $\deg_S(x)$ for $\deg_S(x \otimes 1)$ with $x \otimes 1 \in B \otimes_R S=B_S$.  


\begin{lem} \label{lowersemi}
For any $x \in B$, the map 
\begin{align*}
\Spec R &\to \Z \\ 
\frakp &\mapsto \deg_{R_\frakp}(x)
\end{align*} 
is lower semicontinuous, i.e., for all primes $\frakq \supset \frakp$ we have $\deg_{R_\frakq}(x) \geq \deg_{R_\frakp}(x)$.
\end{lem}

\begin{proof}
Let $n=\deg_R(x)$, and for each integer $0 \leq m \leq n$, let $\fraka_m$ be the ideal of $R$ consisting of all leading coefficients of polynomials $f(T) \in R[T]$ such that $f(x)=0$ with $\deg(f) \leq i$.  Clearly we have $\fraka_0=(0) \subset \fraka_1 \subset \dots \subset \fraka_n=R$.  It follows that $\deg_{R_\frakp}(x_\frakp)=n$ if and only if $\frakp \supset \fraka_{n-1}$, and more generally that $\deg_{R_\frakp}(x_\frakp)=m$ if and only if $\fraka_{m+1} \supsetneq \frakp \supset \fraka_m$, and consequently the map is lower semicontinuous.  
\end{proof}

\begin{cor}
For any $x \in B$ with $\deg_R(x)=n$, the set of primes $\frakp \in \Spec R$ where $\deg_{R_\frakp}(x)=n$ is closed and nonempty.  Moreover, we have $\deg_R(x) \geq \deg_{R_\frakp}(x)$ for all primes $\frakp$.
\end{cor}

\begin{rmk}
Note that if $R[x]$ is projective, Lemma \ref{lowersemi} is immediate since then in fact $\deg_{R_\frakp}(x_\frakp)=\rk(R[x]_\frakp)$ is constant.
\end{rmk}

\begin{defn}
The \emph{degree} of $B$, denoted $\deg_R(B)$ (or simply $\deg(B)$, when no confusion can result), is the smallest positive integer $n \in \Z_{>0}$ such that every element of $B$ has degree at most $n$.
\end{defn}

\begin{exm} \label{deg2exm}
$B$ has degree $1$ as an $R$-algebra if and only if $B=R$.  

If $B$ is free of rank $n$, then $B$ has degree at most $n$ but not necessarily degree $n$, even if $B$ is commutative: for example, the algebra $R[x,y,z]/(x,y,z)^2$ has rank $4$ but has degree $2$ and $R[x,y]/(x^3,xy,y^2)$ has rank $4$ but degree $3$.
\end{exm}

\begin{exm} \label{examprod}
If $K$ is a separable field extension of $F$ with $\dim_F K=n$, then $K$ has degree $n$ as a $F$-algebra (in the above sense) by the primitive element theorem. 

More generally, if $F$ is a field and $B$ is a commutative \'etale algebra with $\#F \geq \dim_F(B)=n$, then $\deg_F(B)=n$.  Indeed, we can write $B \cong \prod_i K_i$ as a product of separable field extensions $K_i/F$, and so if $a_i \in K_i$ are primitive elements with different characteristic polynomials (equivalently, minimal polynomials), which is possible under the hypothesis that $\#K_i \geq \#F \geq n$, then the element $(a_i)_i \in \prod_i K_i \cong B$ has minimal polynomial of degree $n$.  
\end{exm}

\begin{exm}
If $B$ is a central simple algebra over a field $F$, then $\deg(B)^2=\dim_F(B)$.  More generally, if $B$ is a semisimple algebra over $F$, then the degree of $B$ agrees with the usual definition \cite{Lam} given in terms of the Wedderburn-Artin theorem. 
\end{exm}

\begin{defn}
$B$ has \emph{constant degree} $n \in \Z_{>0}$ if $\deg_{R_\frakp}(B_\frakp)=n$ for all prime ideals $\frakp$ of $R$.  
\end{defn}

\begin{exm}
If $R$ is a domain then any $R$-algebra $B$ has constant degree.  Indeed, for any prime $\frakp$ of $R$ we have $\deg_R(B) \geq \deg_{F}(B)$ where $F$ denotes the quotient field of $R$, but on the other hand if $\deg_F(x/d)=n=\deg_F(B)$ for $x \in B$ and $d \in R$, then we must have $\deg_R(x)=n$.
\end{exm}

\begin{lem} \label{constantiscomm}
If $B$ has constant degree $n=\rk_R(B)$, then $B$ is commutative.
\end{lem}

\begin{proof}
We know that $B$ is commutative if and only if $B_\frakm$ is commutative for all maximal ideals $\frakm$ of $B$, since then the commutator $[B,B]$ is locally trivial and hence trivial.  So we may suppose that $R$ is a local ring with maximal ideal $\frakm$.  By hypothesis, we have $\deg_R(B)=n=\rk_R(B)$, so there exists an element $x \in B$ with $\deg_R(x)=n$.  By Nakayama's lemma, we find that $\deg_k(x)=n$, where $k=R/\frakm$ is the residue field of $R$; so the powers of $x$ form a basis for $B_k$, hence also of $B$, and it follows that $B$ is commutative, as claimed.
\end{proof}

\begin{exm}
Lemma \ref{constantiscomm} is false if merely $B$ has degree $n=\rk_R(B)$ (but not constant degree), as in Example \ref{deg32noncomm}.
\end{exm}

Unfortunately, $\deg_R(B)$ is not invariant under base extension, as the following example illustrates.

\begin{exm} \label{notbehavewell}
Let $p$ be prime, let $R=\F_p$, and let $B=\prod_{i=1}^{n} \F_p$ with $n \geq p$.  Then every element $x \in B$ satisfies $x^p=x$, so $\deg_R(B) \leq p$.  On the other hand, the element $x=(0,1,2,\dots,p-1,0,\dots,0)$ has degree $p$ since the elements $1,x,\dots,x^{p-2}$ are linearly independent over $\F_p$ (consider the corresponding Vandermonde matrix), hence $\deg_R(B)=p$.  On the other hand, $\deg_{\overline{\F}_p}(B \otimes_{\F_p} \overline{\F}_p)=n$ by Example \ref{examprod}.
\end{exm}

\begin{defn} \label{gdegdefn} 
The \emph{geometric degree} of $B$, denoted $\gdeg_R(B)$ (or simply $\gdeg(B)$), is the maximum of $\deg_S(B \otimes_R S)$ for all maps $R \to S$ with $S$ a (connected, Noetherian, commutative) ring.  
\end{defn}

\begin{rmk}
In Definition \ref{gdegdefn}, we may equivalently restrict the maximum to rings $S$ which are algebraically closed fields: indeed, if $\gdeg(B)=\deg_S(B \otimes_R S)$ with $\deg_S(x \otimes s)=\deg_S(B \otimes_R S)=n$ then by Lemma \ref{lowersemi} there exists a maximal ideal $\frakm \subset S$ such that $\deg_{S_\frakm}(x \otimes s)=\deg_{k}(x \otimes s)=n$ where $k=S_\frakm/\frakm S_\frakm$, and then $\deg_{\overline{k}}(x \otimes s)=n$ as well, where $\overline{k}$ is the algebraic closure of $k$.
\end{rmk}

For $m \in \Z_{>0}$, we denote by $R[a_1,\dots,a_m]=R[a]$ the polynomial ring in $n$ variables over $R$.

\begin{lem} \label{deguniv}
Suppose that $B$ is generated by $x_1,\dots,x_m$ as an $R$-module, and define 
\[ \xi=a_1x_1+\dots+a_mx_m = \sum_{i=1}^m a_ix_i \in B \otimes_R R[a]. \]
Then $\gdeg_R(B)=\deg_{R[a]}(\xi)<\infty$.
\end{lem}

\begin{proof}
Let $S$ be an $R$-algebra.  Then since $x_1,\dots,x_m$ generate $B \otimes_R S$ as an $S$-algebra, by specialization we see that $\deg_S(B \otimes_R S) \leq \deg_{R[a]}(\xi)$, so $\gdeg(B) \leq \deg_{R[a]}(\xi)$.  But
\[ \deg_{R[a]}(\xi) \leq \deg_{R[a]}(B_{R[a]}) \leq \gdeg(B) \]
by definition, so equality holds.
\end{proof}

We conclude with two results which characterize the geometric degree.

\begin{lem} \label{degflat}
If $S$ is a flat $R$-algebra, then $\gdeg_R(B)=\gdeg_S(B \otimes_R S)$.
\end{lem}

\begin{proof}
For $\xi$ as in Lemma \ref{deguniv}, we have $\gdeg_R(B)=\deg_{R[a]}(\xi)=\rk_{R[a]} R[a][\xi]$; since $S$ is flat over $R$ we have that $S[a]$ is flat over $R[a]$ and $\rk_{R[a]} R[a][\xi] = \rk_{S[a]} S[a][\xi] = \deg_{S[a]}(\xi)=\deg_S(B \otimes_R S)$, as claimed.
\end{proof}

\begin{lem} \label{lowersemimax}
We have $\gdeg_R(B) = \displaystyle{\max_{\frakp \in \Spec R}} \gdeg_{R_\frakp}(B_\frakp)$.  
\end{lem}

\begin{proof}
We have by definition $\gdeg_R(B) \geq \gdeg_{R_\frakp}(B_\frakp)$ for all primes $\frakp$.  Conversely, let $S$ be a ring such that $\gdeg_R(B) = \deg_S(B \otimes S)=n$, and let $x \in B \otimes S$ have $\deg_S(x)=n$.  Then by Lemma \ref{lowersemi}, there exists a prime $\frakq \subset S$ such that $\deg_{S_\frakq}(x)=n$.  If $\frakq$ lies over $\frakp \in \Spec R$, then it follows that $\gdeg_{R_\frakp}(B_\frakp)=n = \gdeg_R(B)$.  The result follows.
\end{proof}

\section{Involutions}

In this section, we discuss the notion of a standard involution on an $R$-algebra, and we compare this to the notion of degree and geometric degree from the previous section.  

\begin{defn}
An \emph{involution (of the first kind)} $\overline{\phantom{x}}:B \to B$ is an $R$-linear map which satisfies:
\begin{enumroman}
\item $\overline{1}=1$, 
\item $\overline{\phantom{x}}$ is an anti-automorphism, i.e., $\overline{xy}=\overline{y}\,\overline{x}$ for all $x,y \in B$, and
\item $\overline{\overline{x}}=x$ for all $x \in B$.  
\end{enumroman}
\end{defn}

If $B^{\text{op}}$ denotes the opposite algebra of $B$, then one can equivalently define an involution to be an $R$-algebra isomorphism $B \to B^{\text{op}}$ such that the underlying $R$-linear map has order at most $2$.

\begin{defn}
An involution $\overline{\phantom{x}}$ is \emph{standard} if $x\overline{x} \in R$ for all $x \in B$.  
\end{defn}

\begin{exm}
The usual adjoint map $M_k(R) \to M_k(R)$ defined by $A \mapsto A^{\dagger}$ (with $AA^{\dagger}=A^{\dagger}A=\det(A)I$) is $R$-linear if and only if $k=2$, since it restricts to the map $r \mapsto r^{k-1}$ on $R$; if $k=2$, then it is in fact a standard involution.  In particular, we warn the reader that many authors consider involutions which are not $R$-linear---although this more general class is certainly of interest (see e.g.\ Knus and Merkurjev \cite{KnusMerkurjev}), we will not consider them here.
\end{exm}

\begin{exm}
To verify that an involution $\overline{\phantom{x}}:B \to B$ is standard, it is not enough to check that $x\overline{x} \in R$ for $x$ in a set of generators for $B$ as an $R$-module.  The Clifford algebra of a quadratic form in many variables gives a wealth of such examples, among others.
\end{exm}

\begin{rmk}
Note that if $\overline{\phantom{x}}$ is a standard involution, so that $x\overline{x} \in R$ for all $x \in B$, then 
\[ (x+1)(\overline{x+1})=(x+1)(\overline{x}+1)=x\overline{x}+x+\overline{x}+1 \in R \] 
and hence $x+\overline{x} \in R$ for all $x \in B$ as well.
\end{rmk}

\begin{exm} \label{trivstddinv}
A standard involution is \emph{trivial} if it is the identity map.  The $R$-algebra $B=R$ has a trivial standard involution as does the commutative algebra $B=R[\eps]/(\eps^2)$ for $R$ any commutative ring of characteristic $2$.

$B$ has a trivial standard involution if and only if $B$ is commutative and $x^2 \in R$ for all $x \in B$.  If the identity map is a standard involution on $B$, then either $B=R$ or $2$ is a zerodivisor in $R$.  Indeed, for any $x \in B$ we have $(x+1)^2 \in R$, so $2x \in R$ for all $x \in B$; if $2$ is a nonzerodivisor in $R$, then $x/1 \in R[1/2]$ so $\rk B[1/2]=\rk B=1$ so $B=R$.  
\end{exm}

Let $\overline{\phantom{x}}:B \to B$ be a standard involution on $B$.  Then we define the \emph{reduced trace} $\trd:B \to R$ by $\trd(x)=x + \overline{x}$ and the \emph{reduced norm} $\nrd:B \to R$ by $\nrd(x)=\overline{x}x$ for $x \in B$.  Since
\begin{equation} \label{stdinvoleq}
x^2-(x+\overline{x})x+\overline{x}x=0
\end{equation}
identically we have $x^2-\trd(x)x+\nrd(x)=0$ for all $x \in B$.  Therefore any $R$-algebra $B$ with a standard involution has $\deg_R(B) \leq 2$.  In particular, for $x,y \in B$ we have
\[ (x+y)^2-\trd(x+y)(x+y)+\nrd(x+y)=0 \]
so
\begin{equation} \label{xyyx}
xy+yx=\trd(y)x+\trd(x)y+\nrd(x+y)-\nrd(x)-\nrd(y).
\end{equation}

An $R$-algebra $S$ is \emph{quadratic} if $S$ has rank $2$.  The following lemma is well-known \cite[I.1.3.6]{Knus}; we give a proof for completeness.

\begin{lem} \label{quad2}
Let $S$ be a quadratic $R$-algebra.  Then $S$ is commutative, we have $\deg_R(S)=\gdeg_R(S)=2$, and there is a unique standard involution on $S$.
\end{lem}

\begin{proof}
First, suppose that $S$ is free.  Then by Lemma \ref{dirsummand}, we can write $S=R \oplus Rx=R[x]$ for some $x \in S$ and so in particular $S$ is commutative.  By Lemma \ref{projfree}, the element $x$ satisfies a unique polynomial $x^2-tx+n=0$ with $t,n \in R$, so $\deg_R(x)=\deg_R(B)=2$.  We define $\overline{\phantom{x}}:R[x] \to R$ by $\overline{x}=t-x$, and we extend the map $\overline{\phantom{x}}$ by $R$-linearity to a standard involution on $S$.  If $\overline{\phantom{x}}:S \to S$ is any standard involution then identically equation (\ref{stdinvoleq}) holds; by uniqueness, we have $t=x+\overline{x}$ and $n=x\overline{x}=\overline{x}x$, and the involution $\overline{x}=t-x$ is unique.  

We now use a standard localization and patching argument to finish the proof.  For any prime ideal $\frakp$ of $R$, the $R_\frakp$-algebra $S_\frakp$ is free.  It then follows that $S$ is commutative, since the map $R$-linear map $S \times S \to S$ by $(x,y) \mapsto xy-yx$ is zero at every localization, hence identically zero.  Further, for each prime $\frakp$, there exists $f \in R \setminus \frakp$ such that $S_f$ is free over $R_f$.  Since $\Spec R$ is quasi-compact, it is covered by finitely many such $\Spec R_f$, and the uniqueness of the involution defined on each $S_f$ implies that they agree on intersections and thereby yield a (unique) involution on $S$.  

To conclude, we must show that $\gdeg_R(S)=2$.  But any base extension of $S$ has rank at most $2$ so has degree at most $2$, and the result follows.
\end{proof}

\begin{rmk}
It follows from Lemma \ref{quad2} that in fact $\nrd(x)=\overline{x}x=x\overline{x}$.
\end{rmk}

By covering any $R$-algebra $B$ with a standard involution by quadratic algebras, we have the following corollary.

\begin{cor} \label{uniqstd}
If $B$ has a standard involution, then this involution is unique.
\end{cor}

\begin{proof}
By localizing at all primes of $R$, we may assume without loss of generality that $B$ is free over $R$.  Choose a basis for $B$ over $R$.  For any element $x$ of this basis, from Corollary \ref{projfreedeg2} we conclude that $S=R[x]$ is free, hence projective; by Example \ref{trivstddinv} (if $S=R$) or Lemma \ref{quad2}, we conclude that $S$ has a unique standard involution.  By $R$-linearity, we see that $B$ itself has a unique standard involution.  
\end{proof}

For the rest of this section, we relate the (geometric) degree of $B$ to the existence of a standard involution.  We have already seen that if $B$ has a standard involution, then it has degree at most $2$.  The converse is not true, as the following example (see also Example \ref{notbehavewell}) illustrates.

\begin{exm} \label{notrlinear}
Let $R=\F_2$ and let $B$ be a Boolean ring of rank at least $3$ over $\F_2$.  Then $B$ has degree $2$, since every element $x \in B$ satisfies $x^2=x$.  The unique standard involution on any subalgebra $R[x]$ with $x \in B \setminus R$ is the map $x \mapsto \overline{x}=x+1$, but this map is not $R$-linear, since 
\[ \overline{x+y}=1+(x+y) \neq \overline{x}+\overline{y}=1+x+1+y=x+y \]
for any $x,y \in B$ such that $1,x,y$ are linearly independent over $\F_2$.  It is moreover not an involution, since if $x \neq y \in B \setminus R$ satisfy $xy \not\in R$, then
\[ \overline{xy}=1+xy \neq \overline{y}\overline{x}=(1+y)(1+x)=1+x+y+xy. \]
\end{exm}

We see from Example \ref{notrlinear} that the condition of $R$-linearity is essential.  We are led to the following key lemma.  

\begin{lem} \label{isantiauto}
Suppose that $B$ has an $R$-linear map $\overline{\phantom{x}}:B \to B$ with $\overline{1}=1$ such that $x\overline{x} \in R$ for all $x \in B$.  Then $\overline{\phantom{x}}$ is a standard involution on $B$.
\end{lem}

\begin{proof}
We must prove that $\overline{\phantom{x}}$ is an anti-involution, i.e., $\overline{xy}=\overline{y}\,\overline{x}$ for all $x,y\in B$.  We can check that this equality holds over all localizations, so we may assume that $B$ is free over $R$.  Since $\overline{\phantom{x}}$ is $R$-linear, we may assume $x,y \in B \setminus R$ are part of an $R$-basis for $B$ which includes $1$.  Write $xy=a+bx+cy+z$ with $z$ linearly independent of $1,x,y$.  Replacing $x$ by $x-c+1$ (again using $R$-linearity), we may assume without loss of generality that $c=1$.  It follows that $1,xy$ belongs to a basis for $B$, so by Corollary \ref{projfreedeg2} we have $R[xy]$ free over $R$.  

Now notice that 
\[ (xy)(\overline{y}\,\overline{x})=x(y\overline{y})\overline{x}=(x\overline{x})(y\overline{y})
=(\overline{y}y)(\overline{x}x)=(\overline{y}\,\overline{x})(xy) \in R \] 
and also (using $R$-linearity one last time)
\[ xy+\overline{y}\,{\overline{x}}=(x+\overline{y})(\overline{x+\overline{y}})-x\overline{x}-y\overline{y} \in R. \] 
But then 
\[ (xy)^2 - (xy + \overline{y}\,\overline{x})xy + (\overline{y}\,\overline{x})(xy) = 0 \]
as well as
\[ (xy)^2 - (xy + \overline{xy})xy + \overline{xy}(xy) = 0 \]
and so by the uniqueness in Lemma \ref{projfree} we conclude that $\overline{xy}=\overline{y}\,\overline{x}$.
\end{proof}

With this lemma in hand, we prove the following central result.

\begin{prop} \label{stdinvolgdeg}
$B$ has a standard involution if and only if $\gdeg_R(B) \leq 2$.
\end{prop}

\begin{proof}
First, suppose that $B$ is free with basis $x_1,\dots,x_m$.  We refer to Lemma \ref{deguniv}; consider the element $\xi=a_1x_1 + \dots + a_mx_m \in B_{R[a]}$, with $R[a]=R[a_1,\dots,a_m]$ a polynomial ring.

The total degree map on $R[a]$ defines a grading of $R[a]$.  We have a natural induced grading on $B_{R[a]}$ as an $R[a]$-module, taking coefficients in the basis $x_1,\dots,x_m$.  Since the coefficients of multiplication in $B_{R[a]}$ are elements of $R$ and so have degree $0$, we see that this grading respects multiplication in $B$.  In this grading, the element $\xi$ has degree $1$.  

Suppose that $\gdeg_R(B) \leq 2$.  The proposition is true if $B=R$, so we may assume $\gdeg_R(B)=2$.  Then $\deg_{R[a]}(\xi)=2$, so there exist polynomials $t(a), n(a) \in R[a]$ such that 
\[ \xi^2-t(a)\xi+n(a)=0. \]
This equality must hold in each degree, so looking in degree $2$ we may assume that $t(a)$ has degree $1$ (and $n(a)$ has degree $2$).  By specialization, it follows that $t(a)$ induces an $R$-linear map $\overline{\phantom{x}}:B \to B$ by $x \mapsto t(x)-x$ with the property that $x\overline{x}=n(x) \in R$ for all $x \in B$.  This map is then a standard involution by Lemma \ref{isantiauto}.

Conversely, suppose that $B$ has a standard involution.  Define the maps (of sets) $t,n:B \to R$ by $\trd(x)=x+\overline{x}$ and $\nrd(x)=x\overline{x}$ for $x \in B$, so that $x^2-\trd(x)x+\nrd(x)=0$ for all $x \in B$.  Define 
\[ t(a)=\sum_{i=1}^{n} \trd(x_i)a_i \in R[a] \]
and
\[ n(a)=\sum_{i=1}^n \nrd(x_i)a_i^2 + \sum_{1\leq i < j \leq n} (\nrd(x_i+x_j)-\nrd(x_i)-\nrd(x_j))a_ia_j \in R[a]. \]
Then $t(a)$ has degree $1$ and $n(a)$ has degree $2$.  Now consider the element 
\begin{equation} \label{haha} 
\xi^2-t(a)\xi+n(a) = \sum_{k=1}^{n} c_k(a) x_k \in B_{R[a]}.
\end{equation}
Each polynomial $c_k(a) \in R[a]$ in (\ref{haha}) has degree $2$.  If we let $e_i$ be the coordinate point $(0,\dots,0,1,0\dots,0)$ with $1$ in the $i$th place for $i=1,\dots,m$, then by construction $c_k(e_i)=c_k(e_i+e_j)=0$ for all $i,j$, and therefore $c_k(a)=0$ identically.  Therefore $\deg_{R[a]}(\xi)=2$ and $\gdeg_R(B)=2$, as claimed.

Now let $B$ be an arbitrary $R$-algebra.  If $\gdeg_R(B) \leq 2$, then by localization and uniqueness (Corollary \ref{uniqstd}) the result follows from the case where $B$ is free.  Conversely, if $B$ has a standard involution, we conclude that $\gdeg_R(B_\frakp) \leq 2$ for all primes $\frakp \in B$.  The result then follows from Lemma \ref{lowersemimax}.
\end{proof}

We conclude this section by relating the existence of a standard involution to degree (not geometric degree).

\begin{prop} \label{deg2trace}
Suppose that $\deg_R(B)=2$ and suppose that there exists $a \in R$ such that $a(a-1)$ is a nonzerodivisor.  Then there is a  standard involution on $B$.
\end{prop}

\begin{proof}
Again by localization and uniqueness, we may suppose that $B$ is free with basis $x_1,\dots,x_m$ with $x_1=1$.  Thus for each $i$, the algebra $S_i=R[x_i]$ is free and by Lemma \ref{quad2} there is a unique standard involution on $S_i$.  This involution extends by $R$-linearity to a map $\overline{\phantom{x}}: B \to B$, which (for the moment) is just an $R$-linear map whose restriction to each $S_i$ is a standard involution.  For $x \in B$, we define $t(x)=x+\overline{x}$ and $n(x)=x\overline{x}$.  

We need to show that in fact $n(x) \in R$ for all $x \in B$, for then $\overline{\phantom{x}}$ is a standard involution by Lemma \ref{isantiauto}.  Let $x,y \in B$ satisfy $n(x),n(y) \in R$.  Since
\begin{align*} 
n(x+y) &= (x+y)(\overline{x+y})=x\overline{x}+y\overline{x}+x\overline{y}+y\overline{y} \\
&= n(x)+n(y)+t(y)x+t(x)y-(xy+yx)
\end{align*}
we have $n(x+y) \in R$ if and only if $xy+yx-t(y)x+t(x)y \in R$, or equivalently if 
\[ (x+y)^2-t(x+y)(x+y) \in R. \] 
By this criterion, it is clear that $n(x+y) \in R$ if and only if $n(ax+by) \in R$ for all $a,b \in R$.  So it is enough to prove that $n(x+y) \in R$ when $1,x,y$ is part of a basis for $B$ with $n(x),n(y) \in R$.  

Let $a \in R$.  By Lemma \ref{projfreedeg2}, since $x+ay$ is contained in a basis for $B$ we have that $R[x+ay]$ is free over $R$ .  Letting $a=1$, we have that $R[x+y]$ is free so $x+y$ satisfies a unique polynomial of degree $2$ over $R$, hence there exists a unique $u \in R$ such that $(x+y)^2-u(x+y) \in R$.  From the above, $n(x+y) \in R$ if and only if $u=t(x+y)$.  

We have
\[ (x+ay)^2 = x^2+a(xy+yx)+a^2y^2 = a(xy+yx) + t(x)x + a^2t(y)y \in B/R \]
and since 
\[ xy+yx = (x+y)^2-x^2-y^2 = u(x+y) -t(x)x - t(y)y \in B/R \]
we have
\[ (x+ay)^2 = (au-at(x)+t(x))x + (au-at(y)+a^2t(y))y  \in B/R. \]
But $\deg_R(B)=2$, so $(x+ay)^2$ is an $R$-linear combination of $1,x+ay$.  But this can only happen if
\[ a(au-at(x)+t(x)) = (au-at(y)+a^2t(y)) \]
which becomes simply
\[ a(a-1)(u-t(x)-t(y))=0. \]
So, if $a(a-1)$ is a nonzerodivisor, then we have $u=t(x)+t(y)=t(x+y)$, as desired.
\end{proof}

We finish then by proving Theorem A.

\begin{cor} \label{cor10}
Suppose that there exists $a \in R$ such that $a(a-1)$ is a nonzerodivisor.  Then the following are equivalent:
\begin{enumroman}
\item $\deg_R(B)=2$;
\item $\gdeg_R(B)=2$;
\item $B \neq R$ and $B$ has a standard involution.
\end{enumroman}
\end{cor}

\begin{proof}
Combine Proposition \ref{stdinvolgdeg} with Proposition \ref{deg2trace} and the trivial implication (ii) $\Rightarrow$ (i).
\end{proof}

\section{Commutative algebras with a standard involution and exceptional rings}

In this section, we investigate two classes of algebras with a standard involution: commutative algebras and exceptional rings.  

First note that if $B$ is a commutative $R$-algebra with a standard involution $\overline{\phantom{x}}:B \to B$, then $\overline{\phantom{x}}$ is in fact an $R$-algebra automorphism.  

\begin{prop} \label{commutativecase}
Let $J=\ann_R(2)=\{x \in R: 2x=0\}$ and let $B$ be a commutative $R$-algebra.  Then $B$ has a standard involution if and only if either $\rk B \leq 2$ or $B$ is generated by elements $x_1,\dots,x_n$ that satisfy $x_i^2 \in J$ for all $i$ and $x_ix_j \in JB$ for all $i \neq j$.
\end{prop}

Consequently, if a commutative $R$-algebra $B$ with $\rk B > 2$ has a standard involution, then the involution is trivial.  

\begin{proof}
Let $B$ be a commutative $R$-algebra with a standard involution and assume that $\rk B > 2$.  

First, suppose that $2=0 \in R$.  Let $1,x,y \in B$ be $R$-linearly independent.  Then by (\ref{xyyx}) we have
\[ 0=2xy=xy+yx=\trd(x)y+\trd(y)x+\nrd(x+y)-\nrd(x)-\nrd(y). \]
Therefore $\trd(x)=\trd(y)=0$.  

Now let $R$ be any commutative ring.  For any $x \in B$ such that $1,x$ is $R$-linearly independent, there exists $y \in B$ such that $1,x,y$ is $R$-linearly independent.  By the preceding paragraph, by considering the image of $x$ in the $R/2R$-algebra $B/2B$ we conclude that $\trd(x)=2u \in 2R$.  Replacing $x$ by $x-u$, we conclude that we may write $B=R \oplus B_0$ where $B_0=\{x \in B : \trd(x)=0\}$.  

Again by (\ref{xyyx}), for any $x,y \in B_0$ such that $1,x,y$ are $R$-linearly independent, we have
\[ 2xy=n=\nrd(x+y)-\nrd(x)-\nrd(y) \in R. \]
But then
\[ x(2xy)=2x^2y=-2\nrd(x)y = nx, \]
and this is a contradiction unless $n=2\nrd(x)=0$.  Thus $2xy=0$ and hence $xy \in JB$, and $x^2=a$ with $a=-\nrd(x) \in J$. 

The conversely is easily verified, equipping $B$ with the trivial standard involution. 
\end{proof}

\begin{cor}
If $2$ is a nonzerodivisor in $R$ and $\rk B > 2$ then $B$ has a standard involution if and only if $B$ is a quotient of the algebra
\[ R[x_1,\dots,x_n]/(x_1,\dots,x_n)^2 \]
for some $n \in \Z_{\geq 2}$.  

If $2=0 \in R$ and $\rk B > 2$ then $B$ has a standard involution if and only if $B$ is a quotient of the algebra 
\[ R[x_1,\dots,x_n]/(x_1^2-a_1,\dots,x_n^2-a_n) \]
for some $n \in \Z_{\geq 2}$ with $a_i \in R$.  
\end{cor}

We now investigate the class of exceptional rings, first defined in the introduction.  Let $M$ be a projective $R$-module $M$ of rank $n-1$ and let $t:M \to R$ be an $R$-linear map.  Then we define the $R$-algebra $B=R \oplus M$ by the rule $xy = t(x)y$ for $x,y \in M$.  This algebra is indeed associative because
\[ (xy)z=(t(x)y)z=t(x)yz=x(yz) \]
for all $x,y,z \in M$ (since $yz=t(y)z \in M$).  The map $\overline{\phantom{x}}:M \to M$ by $x \mapsto t(x)-x$ is an $R$-linear map, and since $x^2=t(x)x$ we have $x\overline{x}=0 \in R$ for all $x \in M$.  We conclude by Lemma \ref{isantiauto} that $\overline{\phantom{x}}$ defines a standard involution on $B$.  

\begin{defn} \label{exceptdefn}
An $R$-algebra $B$ of rank $n$ is an \emph{exceptional ring} if there is a left ideal $M \subset B$ such that $B=R \oplus M$ and the map $M \to \Hom_R(M,B)$ given by left multiplication factors through a linear map $t:M \to R$.
\end{defn}

It follows from the preceding paragraph that an exceptional ring has a standard involution.  Since a standard involution is necessarily unique (Corollary \ref{uniqstd}), if $B=R \oplus M$ is exceptional, corresponding to $t:M \to R$, then we automatically have $t=\trd|_M$.  If $R \to S$ is a ring homomorphism and $B$ is an exceptional ring over $R$ then $B \otimes_R S$ is an exceptional ring over $S$.

\begin{exm} \label{quadraticexcept}
Suppose $B$ is a quadratic $R$-algebra.  Then $B$ is a free exceptional ring if and only if $B \cong R \times R$ or $B \cong R[x]/(x^2)$.  Moreover, the splitting $B=R \oplus M$ (with $\rk_R(M)=1$) is unique up to replacing $M$ by its conjugate $\overline{M}$.
\end{exm}

\begin{lem} \label{canonsplitBRM}
If $B$ is an exceptional ring and $\rk(B)>2$, then the splitting $B=R \oplus M$ is unique.
\end{lem}

\begin{proof}
Localizing, we may assume that $B$ and $M$ are free as $R$-modules.  Suppose that $B=R \oplus M = R \oplus M'$ are splittings associated to linear maps $t:M\to R$ and $t':M' \to R$.  Let $x,y \in M$ be such that $1,x,y$ are $R$-linearly independent.  Then $x=r+x'$ and $y=s+y'$ with $r,s \in R$ and $x',y' \in M'$.  From $xy=t(x)y$ we have
\[ (r+x')(s+y')=rs+sx'+(r+t'(x'))y' = t(x)(s+y')=t(x)s + t(x)y'. \]
Since $1,x',y'$ are $R$-linearly independent, we conclude from the coefficient of $x'$ that $s=0$ and hence $y=y'\in N$.  Interchanging the roles of $x$ and $y$ we find $x=x' \in N$ as well.  
\end{proof}

\begin{rmk} \label{equivexcept}
Consequently, there is an equivalence of categories between the category of exceptional rings of rank $n>2$, with morphisms isomorphisms, and the category of $R$-linear maps $t:M \to R$ with $M$ projective of rank $n-1>1$, where a morphism between $t:M \to R$ and $t':M' \to R$ is simply a map $f: M \to M'$ such that $t' \circ f = t$.  
\end{rmk}

We will show in the next section that if $\rk B=3$ and $B$ has a standard involution then $B$ is exceptional.

\begin{lem} \label{exceptlocal}
An $R$-algebra $B$ with $\rk(B)>2$ is exceptional if and only if $B_\frakp$ is exceptional for all primes $\frakp$ of $R$.
\end{lem}

\begin{proof}
If $B$ is exceptional then obviously $B_\frakp$ is exceptional for all primes $\frakp$.  Conversely, suppose $B_\frakp$ is exceptional for all primes $\frakp$ of $R$.  By Lemma \ref{canonsplitBRM}, we may write $B_\frakp = R_\frakp \oplus M_\frakp$ uniquely for each prime $\frakp$.  Gluing, we have $B=R \oplus M$ where
\[ M = \{x \in B : x_\frakp \in M_\frakp \text{ for all $\frakp$}\}. \]
Similarly, by uniqueness the linear maps $t_\frakp:M_\frakp \to R_\frakp$ glue to give a map $t:M \to R$ such that $xy=t(x)y$ for all $x,y \in M$.  
\end{proof}

\begin{rmk} 
Lemma \ref{exceptlocal} is false when $\rk(B)=2$, consequent to the fact that there exists a ring $R$ and an element $a \in R$ such that $a$ is a square in every localization $R_\frakp$ but $a$ itself is not a square: the algebra $B=R[x]/(x^2-a)$ is then a counterexample.  
\end{rmk}

Exceptional rings can be distinguished by a comparison of minimal and characteristic polynomials.  For an element $x \in B$, let $\mu(x;T)=T^2-\trd(x)T+\nrd(x)$ and let $\chi_L(x;T)$ (resp.\ $\chi_R(x;T)$) be the characteristic polynomial of left (resp.\ right) multiplication as in Lemma \ref{charpoly}.  Recall from Section 1 that if $x \not\in R$, then $\mu(x;T)$ is the polynomial which realizes $\deg_R(x)=2$, i.e., it is the monic polynomial of smallest degree with coefficients in $R$ which is satisfied by $x$.  Let $\Tr(x)$ denote the trace of left multiplication by $x$.  

\begin{lem}
Let $B=R \oplus M$ be an exceptional ring.  Then for all $x \in M$, we have $\mu(x;T)=T(T-\trd(x))$ and
\[ \chi_L(x;T)=T(T-\trd(x))^{n-1}=\mu(x;T)(T-\trd(x))^{n-2} \] 
so $\Tr(x)=(n-1)\trd(x)$ and 
\[ \chi_R(x;T)=T^{n-1}(T-\trd(x)). \]
\end{lem}

\begin{proof}
This statement follows from a direct calculation.
\end{proof}

\section{Algebras of rank $3$}

We saw in Section 2 that an algebra of rank $2$ is necessarily commutative, has (geometric and constant) degree $2$, and has a (unique) standard involution.  Quadratic $R$-algebras are classified by their discriminants, and this is a subject that has seen a great deal of study (see Knus \cite{Knus}).  In this section, we consider the next case, algebras of rank $3$.

First, suppose that $B$ is a free $R$-algebra of rank $3$.  We follow Gross and Lucianovic \cite[\S 2]{GrossLuc} (see also Bhargava \cite{BhargavaCubic}).  They prove that if $B$ is commutative and $R$ is a PID or a local ring, then $B$ has an $R$-basis $1,i,j$ such that
\begin{align*}
i^2 &= -ac+bi-aj \\
j^2 &= -bd+di-cj \tag{$C$} \\
ij &= -ad
\end{align*}
with $a,b,c,d \in R$.  But upon further examination, we see that their proof works for free $R$-algebras $B$ over an arbitrary commutative ring $R$ and that their calculations remain valid even when $B$ is noncommutative, since they use only the associative laws.  If we write 
\[ ji=r+si+tj \] 
with $r,s,t \in R$, then the algebra ($C$) is associative if and only if 
\begin{equation} \label{rst}
as=dt=0 \quad \text{and} \quad r+ad=-bs=ct.
\end{equation}
For example, $B$ is commutative if and only if $r=-ad$ and $s=t=0$.

We now consider the classification of such algebras $B$ by degree.  We assume that $B$ has constant degree (otherwise see Example \ref{deg32noncomm}).  If $\deg_R(B)=3$, then $B$ is commutative by Lemma \ref{constantiscomm}.  So we are left to consider the case $\deg_R(B)=2$.  Then the coefficients of $j,i$ in $i^2,j^2$, respectively, must vanish, so $a=d=0$ in the laws ($C$), and we have $r=-bs=ct$ in (\ref{rst}).  After the equivalences of Theorem A, it is natural to consider the case where further $B$ has a standard involution.  Then 
\[ 0 = -ad = \overline{ij} = \overline{j}\,\overline{i}=(-c-j)(b-i)=-bc+ci-bj+ji \]
so $ji=bc-ci+bj$ and $r=bc$, $s=-c$, $t=b$.  Now replacing $i$ by $\overline{i}=b-i$, and letting $u=b$ and $v=-c$ we obtain the equivalent multiplication rules
\begin{equation*}
\begin{aligned}
i^2 &= ui & \qquad\qquad ij &= uj \\ 
j^2 &= vj & \qquad\qquad ji &= vi.
\end{aligned}
\tag{$NC$} 
\end{equation*}
Following Gross and Lucianovic, we call such a basis $1,i,j$ a \emph{good basis}.  Note that by definition an algebra with multiplication rules ($NC$) is exceptional, with $M=Ri \oplus Rj$.  We have therefore proven that every free $R$-algebra $B$ of rank $3$ with a standard involution is an exceptional ring.

We have shown that there is a bijection between pairs $(u,v) \in R^2$ and free $R$-algebras of rank $3$ with a standard involution equipped with a good basis.  The natural action of $GL_2(R)$ on a good basis, defined by
\begin{equation} \label{ijipjp}
\begin{pmatrix}
i \\ j \end{pmatrix} \mapsto \begin{pmatrix}
i' \\ j'
\end{pmatrix} = \begin{pmatrix}
\alpha & \beta \\
\gamma & \delta 
\end{pmatrix}
\begin{pmatrix}
i \\ j \end{pmatrix}
\end{equation}
takes one good basis to another, and the induced action on $R^2$ is simply $(u,v) \mapsto (\alpha u + \beta v, \gamma u + \delta v)$.  Therefore the set of good bases of $B$ is a principal homogeneous space for the action of $GL_2(R)$, and we have proved the following.

\begin{prop} \label{rk3tMR}
Let $N$ be a free module of rank $2$.  Then there is a bijection between the set of orbits of $GL(N)$ acting on $N$ and the set of isomorphism classes of free $R$-algebras of rank $3$ with a standard involution.
\end{prop}

\begin{exm}
The map $R^2 \to R$ with $e_1,e_2 \mapsto u,v$ corresponds to the algebra $(NC)$.  In particular, the zero map $R^2 \to R$ corresponds to the commutative algebra $R[i,j]/(i,j)^2$.  
\end{exm}

\begin{rmk}
The universal element $\xi=x+yi+zj$ of the algebra $B$ defined by the multiplication rules $(NC)$ for $u,v \in R$ satisfies the polynomial
\[ \xi^2-(2x+uy+vz)\xi + (x^2+uxy+vxz) = 0 \]
hence $\gdeg_R(B)=2$ and this verifies (in another way) that any such algebra indeed has a standard involution.  

The only algebra which is both of type $(C)$ and $(NC)$ is the algebra with $u=v=0$ (or $a=b=c=d=0$), i.e., the commutative algebra $R[i,j]/(i,j)^2$.  
\end{rmk}

\begin{exm} \label{deg32noncomm}
We pause to exhibit in an explicit example the irregular behavior of an algebra which is not of constant degree.  Roughly speaking, we can glue together an algebra of degree $2$ and an algebra of degree $3$ along a degenerate algebra of degree $3$.

Let $k$ be a field and let $R=k[a,b]/(ab)$, so that $\Spec R$ is the variety of intersecting coordinate lines in the (affine) plane.  Consider the free $R$-algebra $B$ with basis $1,i,j$ and with multiplication defined by
\begin{equation*}
\begin{aligned}
i^2 &= bi-aj & \qquad\qquad ij &= -a^2 \\
j^2 &= ai-bj & \qquad\qquad ji &= b^2-a^2-bi+bj.
\end{aligned}
\end{equation*}
We note that $B$ indeed has degree $3$, since for example $i^3=b^2i + a^3$ is the monic polynomial of smallest degree satisfied by $i$.  

We have $R_{(b)} \cong k(a)$ with $B_{(b)}$ isomorphic to the algebra above with $b=0$; this algebra is commutative of rank 3, with $ij=ji=-a^2$ (and $i^2=-aj$ and $j^2=ai$).  On the other hand, we have $R_{(a)} \cong k(b)$ with $B_{(a)}$ subject to $ij=0 \neq b^2 - bi + bj = ji$ and $i^2=bi$, $j^2=-bj$, so $B_{(b)}$ is a noncommutative algebra of rank 3 and degree $2$.
\end{exm}

Now consider a (projective, not necessarily free) $R$-algebra $B$ of rank $3$ with a standard involution.  

\begin{lem} \label{canonicalsplitting3}
There exists a unique splitting $B = R \oplus M$ with $M$ projective of rank $2$ such that for all primes $\frakp$ of $R$ and any basis $i,j$ of $M_\frakp$, the elements $1,i,j$ are a good basis for $B_\frakp$.
\end{lem}

\begin{proof}
Let $M$ be the union of all subsets $\{i,j\} \subset B$ such that $i,j$ satisfy multiplication rules as in $(NC)$.  We claim that $B=R \oplus M$ is the desired splitting.  It suffices to show this locally, and for any prime $\frakp$, the module $M_\frakp$ contains all good bases for $B_\frakp$ by the calculations above, and the result follows.
\end{proof}

Let $B=R \oplus M$ as in Lemma \ref{canonicalsplitting3}.  Consider the map
\[ M \to \End_R(M). \]
According the multiplication laws $(NC)$, this map is well-defined and factors as $M \to R \subset \End_R(M)$ through scalar multiplication, since it does so locally.  It follows by definition that $B$ is an exceptional ring, and that the splitting $B=R \oplus M$ agrees with that in Lemma \ref{canonsplitBRM}.

\begin{thm} \label{rank3}
Every $R$-algebra $B$ of rank $3$ with a standard involution is an exceptional ring.  There is an equivalence of categories between the category of $R$-algebras $B$ of rank $3$ with a standard involution and the category of $R$-linear maps $t:M \to R$ with $M$ projective of rank $2$.
\end{thm}

\begin{cor}
There is a bijection between the set of isomorphism classes of $R$-algebras of rank $3$ with a standard involution and isomorphism classes of $R$-linear maps $t:M \to R$ with $M$ projective of rank $2$.  
\end{cor}

We conclude this section with the following observation.  Consider now the \emph{right} multiplication map $M \to \End_R(M)$.  When $M=R^2$ is free as in $(NC)$ with basis $i,j$, we have under this map that
\[ i \mapsto \begin{pmatrix} u & 0 \\ v & 0 \end{pmatrix}, \quad j \mapsto \begin{pmatrix} 0 & u \\ 0 & v \end{pmatrix}. \]
If $\ann_R (u,v) = (0)$, then this map is injective.  Note that $(u,v) = t(R^2) \subset R$, and $\ann(u,v)=(0)$ if and only if $B_\frakp$ is noncommutative for every prime ideal $\frakp$, in which case we say $B$ is \emph{noncommutative everywhere locally}.  We compute directly that element $k=vi-uj$ satisfies $k^2=0$, and hence is contained in the Jacobson radical of $B$.  Indeed, we have $ki=kj=0$, and of course $ik=uk$ and $jk=vk$.  In any change of good basis as in (\ref{ijipjp}), we find that $k'=(\alpha\delta-\beta\gamma)k$ with $\alpha\delta-\beta\gamma \in R^*$, so the $R$-module (or even two-sided ideal) generated by $k$ is independent of the choice of good basis, and so we denote it $J(B)$.  Note that $J(B)$ is free if and only if $\ann_R(u,v) = (0)$.

More generally, suppose that $t:M \to R$ has $\ann_R t(M) = (0)$, or equivalently that $B$ is noncommutative everywhere locally.  Then the right multiplication map is injective since it is so locally, and so the right multiplication map yields an injection $B \hookrightarrow \End_R(M)$.  By the above calculation, we see that two-sided ideals $J(B_\frakp)$ for each prime $\frakp$ patch together to give a well-defined two-sided ideal $J(B)$ of $B$ which is projective of rank $1$, and the image of $B$ in $\End_R(M)$ annihilates this rank $1$ submodule.  Conversely, given a flag $I \subset J$, we associate the subalgebra $B = R \oplus M$ where $M \subset \End_R(I \subset J)$ (acting on the right) consists of elements which annihilate $I$.  We obtain the following proposition.

\begin{prop}
There is a bijection between the set of isomorphism classes of $R$-algebras of rank $3$ with a standard involution which are noncommutative everywhere locally and flags $I \subset J$ such that $I,J$ are projective of ranks $1,2$.  
\end{prop}

\begin{exm}
If $M=R^2 \to R$ is the map $e_1 \mapsto 1$ and $e_2 \mapsto 0$, then the above correspondence realizes the associated algebra $B$ as isomorphic to the upper-triangular matrices in $M_2(R)$.
\end{exm}


\end{document}